\documentclass{article}
\usepackage{amsmath, amsthm, amsfonts, amssymb, amscd, stmaryrd,xcolor,enumerate, enumitem}

\usepackage{graphicx}
\usepackage{tikz}

\theoremstyle{definition}

\theoremstyle{plain}
\newtheorem{Thm}{Theorem}
\newtheorem{Lem}{Lemma}
\newtheorem{Cor}{Corollary}

\newtheorem{Prop}{Proposition}

\numberwithin{equation}{section}	
\numberwithin{figure}{Prob}
\numberwithin{table}{Prob}
\numberwithin{Thm}{section}
\numberwithin{Lem}{section}
\numberwithin{Def}{section}
\numberwithin{Prop}{section}
\numberwithin{Cor}{section}

\newcommand{\Z}{\mathbb{Z}}

\newcommand{\C}{\mathbb{C}}

\newcommand{\Span}{\text{Span}}

\title{Almost Commutative Terwilliger Algebras of Group Association Schemes I: Classification}
\author{Nicholas L. Bastian, Stephen P. Humphries}

\date{\today}

\begin{document}

\maketitle

\begin{abstract}
Terwilliger algebras are a subalgebra of a matrix algebra that are constructed from association schemes over finite sets. In 2010, Rie Tanaka defined what it means for a Terwilliger algebra to be almost commutative. In that paper she gave five equivalent conditions for a Terwilliger algebra to be almost commutative. In this paper, we provide a classification of which groups result in an almost commutative Terwilliger algebra when looking at the group association scheme (the Schur ring generated by the conjugacy classes of the group). In particular, we show that all such groups are either abelian, or Camina groups. Following this classification, we then compute the dimension and non-primary primitive idempotents for each Terwilliger algebra of this form for the first three types of groups whose group association scheme gives an almost commutative Terwilliger algebra. The final case will be considered in a second paper.    \end{abstract}



\section{Introduction}

Terwilliger algebras were originally developed in the 1990's by Paul Terwilliger to provide a method for studying commutative association schemes. In particular, he looked at P-Polynomial and Q-Polynomial association schemes. Over the course of the three papers \cite{Terwilliger1,Terwilliger2,Terwilliger3} in which Terwilliger algebras were originally introduced, Terwilliger manages to find a combinatorial characterization of thin P-Polynomial and Q-Polynomial association schemes. 

Since Terwilliger algebras were defined by Terwilliger, others have studied them with relation to various commutative association schemes. Terwilliger algebras have been studied as they relate to Johnson schemes in \cite{TerJohnsonScheme1, TerJohnsonSchemeLeon, TerJohnsonScheme2, TerJohnsonSchemeGroupRep}. Extending this study of Johnson schemes, the Johnson geometry was studied using Terwilliger algebras in \cite{TerJohnsonGraph1, TerJohnsonGraph2}. The Terwilliger algebra for Hamming schemes was studied in \cite{TerHamming} and the incidence matrix of the Hamming graph was studied in \cite{HammingGraph}. Another common place where Terwilliger algebras are used is in looking at distance regular graphs, which also give rise to a P-polynomial association scheme: see \cite{DRG4, DRG3, DRG2, DRG5, DRG1}. In addition to distance regular graphs, Terwilliger algebras have also been applied to bipartite distance-regular graphs in \cite{BDRG1, BDRG2, BDRG3}, almost-bipartite distance-regular graphs in \cite{ABDRG1}, and distance-biregular graphs in \cite{DBRG1}. Terwilliger algebras have also been studied for group association schemes in \cite{GroupAssoc1, Bannaiarticle, Bastian, GroupAssoc2}. In addition to these commonly studied areas, Terwilliger algebras have also been studied with relation to wreath products of association schemes in \cite{WreathProduct1, WreathProduct2}, the hypercube in \cite{Hypercube}, latin squares in \cite{LatinSquare}, odd graphs in \cite{Oddgraph}, a strongly regular graph in \cite{SRG}, quantum adjacency algebras in \cite{Quantum}, almost-bipartite P- and Q-polynomial association schemes in \cite{ABPQ}, Lee association schemes over $\Z_4$ in \cite{Leescheme}, cyclotomic association schemes in \cite{CyclAssoc}, semidefinite programming in \cite{Semiprog}, and generalized quadrangles in \cite{Quadrangles}. Generalizations of Terwilliger algebras can be found in \cite{GenTer,Nicholson, VarietalTer}.

This paper is organized as follows. Section $2$ will discuss what a Terwilliger algebra is as well as what it means for a Terwilliger algebra to be almost commutative. In Section $3$ we shall discuss Camina groups as well as results relating to them that are relevant to the classification. In Section $4$ we will discuss the classification of Almost Commutative Terwilliger algebras for group association schemes. Then in Section $5$ we shall determine the non-primary primitive idempotents of Almost Commutative Terwilliger algebras for group association schemes in three of the four cases. The remaining case is considered in a subsequent paper \cite{Bastian2}. 

We note all computations considered in the preparation of this paper were preformed in Magma \cite{Magma}.

\section{Terwilliger Algebras}

In this section we shall discuss some results relating to Terwilliger algebras. For more information about them see \cite{Terwilliger1}. Let $\Omega$ be a finite set. Suppose $A_0,A_1,\cdots, A_d\in M_{|\Omega|}(\C)$ are $0,1$ matrices with rows and columns indexed by the elements of $\Omega$. Let $A^t$ denote the transpose of the matrix $A$. If additionally 
\begin{enumerate}
    \item $A_0=I_{|\Omega|}$;
    \item for all $i=1,2,\dots, d$, we have $A_i^t=A_j$ for some $j=1,2,\dots ,d$;
    \item for all $i,j\in \{1,2,\dots, d\}$ we have $A_iA_j=\sum_{k=1}^d p_{ij}^k A_k$;
    \item none of the $A_i$ is equal to $0_{|\Omega|}$ and $\sum_{i=0}^d A_i$ is the all $1$ matrix;
\end{enumerate}
then we say that $\mathcal{A}=(\Omega,\{A_i\}_{0\leq i\leq d})$ is an \emph{association scheme}. We call the matrices $A_0,A_1,\cdots, A_d$ the \emph{adjacency matrices} of the association scheme and the constants $p_{ij}^k$ are called the \emph{intersection numbers}. If we have $A_iA_j=A_jA_i$ for all $i,j$, then we say that we have a \emph{commutative association scheme}. 

The association schemes that we will be interested in come from a Schur ring, which is an example of an association scheme and which we now define. Let $G$ be a group and $R$ be a commutative ring. For any $C\subseteq G$, we let $\overline{C}=\sum_{g\in C} g\in R[G]$. We call this a \emph{simple quantity}. Given any $C\subseteq G$, we let $C^*=\{g^{-1}\colon g\in G\}$.

Now given a group $G$ and ring $R$, let $\mathcal{P}$ be a partition of $G$ of finite support, that is, if $C\in \mathcal{P}$ then $|C| < \infty$. We say that $\mathfrak{S}=\Span_R\{\overline{C}|C\in \mathcal{P}\}$ is a \emph{Schur ring} if:
\begin{enumerate}
    \item $\{e\} \in \mathcal{P};$
    \item if $C \in \mathcal{P}$, then $C^* \in \mathcal{P};$
    \item for all $C, D \in \mathcal{P}$,
$\overline{C}\cdot\overline{D} = \sum_{E\in \mathcal{P}} \lambda_{CDE}\overline{E}$,
where all but finitely many $\lambda_{CDE}=0$.
\end{enumerate}

Schur rings have been studied on their own extensively; see \cite{Me2019Infinite,Ma89,Wielandt49}. They give rise to association schemes in the following way. Given a Schur ring $\mathfrak{S}$ over the finite group $G$ with partition set $\mathcal{P}=\{P_0=\{e\},P_1,P_2,\cdots , P_d\}$, we define $|G|\times |G|$ matrices $A_i$, $0\leq i\leq d$, by
\[(A_i)_{xy}=\left\{\begin{array}{cc}
    1 & \text{ if }yx^{-1}\in P_i  \\
    0 & \text{ otherwise.} 
\end{array}\right.\]
With this construction $(G,\{A_i\}_{0\leq i\leq d})$ is an association scheme. If the Schur ring is commutative then the association scheme is commutative. The structure constants $\lambda_{ijk}$ coming from $\overline{P_i}\cdot \overline{P_j}=\sum_{P_j\in \mathcal{P}}\lambda_{ijk}\overline{P_k}$ are the intersection numbers of the association scheme. In particular, partitioning $G$ by the conjugacy classes of $G$, we get a Schur ring. This in turn gives rise to a commutative association scheme, called the \emph{group association scheme}. 

Returning to arbitrary association schemes, we notice that $\mathfrak{A}=\Span_\C\{A_0,A_1,\cdots, A_d\}$ is an algebra called the \emph{Bose-Mesner algebra} of the association scheme.

We let $E_0,E_1,\dots, E_d$ be the primitive idempotents of $\mathfrak{A}$. As $\mathfrak{A}$ is closed under component-wise multiplication (Hadamard product, denoted $\circ$), we have $E_i\circ E_j\in \mathfrak{A}$. Then there are complex numbers $q_{ij}^k$, called the \emph{Krein parameters}, such that
\[E_i\circ E_j=\frac{1}{|\Omega|}\sum_{k=0}^d q_{ij}^kE_k.\]

We further define diagonal matrices $E_i^*(x)$ \emph{with base point $x\in \Omega$} by the rule
\[(E_i^*(x))_{yy}=\Bigg\{ \begin{array}{cc}
    1  & (x,y)\in R_i \\
    0 & \text{otherwise}.
\end{array}\]
Essentially what we are doing here is taking the $x$th row of $A_i$ and putting it along the main diagonal to form $E_i^*(x)$. We call $\mathfrak{A}^*(x)=\Span_\C\{E_0^*(x),E_1^*(x),\cdots E_d^*(x)\}$ the \emph{dual Bose-Mesner Algebra} of $\mathcal{A}$ with respect to $x$.  

Now we define the \emph{Terwilliger algebra} with base point $x$ of the association scheme $\mathcal{A}=(\Omega,\{A_i\}_{0\leq i\leq d})$ to be the subalgebra of $M_{|\Omega|}(\C)$ generated by $\mathfrak{A}$ and $\mathfrak{A}^*(x)$. We denote this by $T(x)$. For the group association scheme we always take the base point to be the identity $e$ in the group. In this case we let $T(G)$ denote the Terwilliger algebra for the group association scheme.

Let
\[T_0(x)=\Span_\C \{E_i^*(x)A_jE_k^*(x)\colon 0\leq i,j,k\leq d\}.\]
From \cite{Bannaiarticle} we have $\dim T_0(x)=|\{(i,j,k)\colon p_{ij}^k\neq 0\}|$. We shall say that a Terwilliger algebra is \emph{triply regular} if $T(x)=T_0(x)$.

If $|\Omega|>1$, then the Terwilliger algebra $T(x)$ is non-commutative and semi-simple for all $x\in \Omega$ \cite{Terwilliger1}. As such, it has a Wedderburn decomposition. Given an association scheme $(\Omega,\{A_i\}_{0\leq i\leq d})$, the Terwilliger algebra $T(x)$ for this association scheme always has an irreducible ideal of dimension $(d+1)^2$; see \cite{WreathProduct1}. This Wedderburn component is called the \emph{primary component} and will be denoted by $V$.

Closely related to the Wedderburn decomposition of the Terwilliger algebra are those $T(x)-$modules that are irreducible when $T(x)$ acts on $\C^{|\Omega|}$ by left multiplication. We note that with this action $\C^{|\Omega|}$ decomposes into a direct sum of irreducible $T(x)-$ modules. Letting $\hat{x}\in \C^{|\Omega|}$ have $1$ in position $x$ and $0$ elsewhere, it is shown in \cite{Terwilliger1} that $\mathfrak{A}\hat{x}$ is an irreducible $T(x)-$module of dimension $d+1$. We call this irreducible module the \emph{primary module}.

We now say a Terwilliger algebra $T(x)$ is \emph{almost commutative} if every non-primary irreducible $T(x)-$module is $1-$ dimensional. There is a classification of such Terwilliger algebras when the underlying association scheme is commutative:

\begin{Thm}[Tanaka, \cite{Tanaka}]\label{thm:tanaka}  Let $\mathcal{A}=(\Omega,\{R_i\}_{0\leq i\leq d})$ be a commutative association scheme. Let $T(x)$ be the Terwilliger algebra of $\mathcal{A}$ for some $x\in \Omega$. The following are equivalent:
    \begin{enumerate}
        \item Every non-primary irreducible $T(x)-$module is $1-$ dimensional for some $x\in \Omega$.
        \item Every non-primary irreducible $T(x)-$module is $1-$ dimensional for all $x\in \Omega$.
        \item The intersection numbers of $\mathcal{A}$ have the property that for all distinct $h,i$ there is exactly one $j$ such that $p_{ij}^h\neq 0$ $(0\leq h,i,j\leq d)$.
        \item The Krein parameters of $\mathcal{A}$ have the property that for all distinct $h,i$ there is exactly one $j$ such that $q_{ij}^h\neq 0$ $(0\leq h,i,j\leq d)$.
        \item $\mathcal{A}$ is a wreath product of association schemes $\mathcal{A}_1,\mathcal{A}_2,\cdots, \mathcal{A}_n$ where each $\mathcal{A}_i$ is either a $1-$class association scheme or the group scheme of a finite abelian group.
    \end{enumerate} 
Moreover, a Terwilliger algebra satisfying these equivalent conditions is triply regular.
\end{Thm}

 For the definition of the wreath product of association schemes see \cite{WreathProd}. We note that a Terwilliger algebra for a commutative association scheme is almost commutative if it satisfies the five equivalent conditions in Theorem \ref{thm:tanaka}.
 The focus of this paper is to determine exactly for which groups the Terwilliger algebra of the group association scheme is almost commutative. The main result is Theorem \ref{thm:acclassify}. As all group association schemes are commutative, we note that Theorem \ref{thm:tanaka} applies to the Terwilliger algebras under consideration.

\section{Camina Groups}

Camina groups were first studied by Camina in \cite{Camina} as a generalization of Frobenius groups and extra special groups. Results on Camina groups can be found in the survey \cite{Lewis1}.

A group $G$ is a \emph{Camina group} if every conjugacy class of $G$ outside of $G'$ is a coset of $G'$ in $G$. An important result here is:

\begin{Thm}[Dark and Scoppola, \cite{Dark1}] Let $G$ be a Camina group. Then one of the following is true:
        \begin{itemize}[leftmargin=*]
            \item $G$ is a Frobenius group whose Frobenius complement is cyclic.
            \item $G$ is a Frobenius group whose Frobenius complement is $Q_8$, the quaternion group of order $8$.
            \item $G$ is a $p-$group for some prime $p$ of nilpotency class $2$ or $3$.
        \end{itemize}
\end{Thm}

We will also need:

\begin{Prop}[Macdonald, \cite{MacDonald}]\label{thm:Camina3}
        Let $G$ be a Camina $p-$group of nilpotency class $3$. Then the lower central series of $G$ is $\{e\}\leq Z(G)\leq G'\leq G$. Furthermore, $[G\colon G']=p^{2n}$ and $[G'\colon Z(G)]=p^n$ for some even integer $n$. Additionally, $G/G'$ and $G'/Z(G)$ are elementary abelian groups.
\end{Prop}

\begin{Prop}[Macdonald, \cite{MacDonald}]\label{prop:cam2center}
        Let $G$ be a Camina p-group. Then $Z(G)$ is an elementary abelian group.
    \end{Prop}

    \begin{Lem}[\cite{Lewis1}]\label{lem:dercam}
            Let $G$ be a non-abelian Camina group. Then $Z(G)\leq G'$.
        \end{Lem}

        \begin{Prop}[\cite{Lewis1}]\label{prop:cammod}
        Suppose that $G$ is a Camina group and $N\trianglelefteq G'$. Then $G/N$ is a Camina group.
    \end{Prop}

Another result that will be of particular interest to us is:

\begin{Prop}[Macdonald, Theorem 5.2(i), \cite{MacDonald}]\label{lem:pcamclass}
    Let $G$ be a Camina $p-$group. Let $\gamma_2(G)=[G,G]$ and $\gamma_3(G)=[\gamma_2(G),G]=Z(G)$ be the second and third terms in the lower central series of $G$. Then
        \[x^G=x\gamma_2(G)\text{ if }x\in G\setminus \gamma_2(G),\]
        \[x^G=x\gamma_3(G)\text{ if }x\in \gamma_2(G)\setminus \gamma_3(G),\]
        \[x^G=\{x\} \text{ if }x\in \gamma_3(G).\]
\end{Prop}

A result for Camina groups that are also Frobenius groups that we use is:

\begin{Thm}[Cangelmi and Muktibodh, \cite{Few_Conj}]\label{thm:twoclass}
     Let $G$ be a finite group. Then, $G$ is a Camina group and $G'$ is the union of two conjugacy classes if and only if either $G$ is a Frobenius group with Frobenius kernel $\Z_p^r$ and Frobenius complement $\Z_{p^r-1}$, or $G$ is an extra-special $2-$group.
\end{Thm}

Generalizing a  Camina group we have for $H\leq G$, the pair $(G,H)$ is a \emph{Camina pair} if for all $g\in G\setminus H$, $g$ is conjugate to every element of $gH$. Note $(G,G')$ is a Camina pair if and only if $G$ is a Camina group. More information on Camina pairs can be found in \cite{Lewis1}.

Our interest in Camina pairs stems from the following result about equivalent conditions to $(G,H)$ being a Camina pair. This result was proven over the course of several papers \cite{Camina, Chillag1, Kuisch, Yongcai}. 

\begin{Lem}\label{lem:pair}
    Let $1<K\triangleleft G$. Then the following are equivalent:
    \begin{enumerate}
        \item $(G,K)$ is a Camina pair.
        \item If $x\in G\setminus K$, then $|C_G(x)|=|C_{G/K}(xK)|$.
        \item If $xK$ and $yK$ are conjugate and nontrivial in $G/K$, then $x$ is conjugate to $y$ in $G$.
        \item If $C_1=\{e\},C_2,\cdots , C_m$ are the conjugacy classes of $G$ contained in $K$ and $C_{m+1},\cdots, C_n$ are the conjugacy classes of $G$ outside $K$, then $C_iC_j=C_j$ for $1\leq i\leq m$ and $m+1\leq j\leq n$.
    \end{enumerate}
\end{Lem}

\section{Almost Commutative Terwilliger Algebras for Group Association Schemes}

For the remainder of this paper we let $T(G)$ to be the Terwilliger algebra for the group association scheme of the group $G$ with the base point being the identity of the group. Let $C_0=\{e\},C_1,C_2,\cdots, C_d$ be the conjugacy classes of $G$. We shall show the following:

\begin{Thm}\label{thm:acclassify}
        Let $G$ be a finite group. Then $T(G)$ is almost commutative if and only if $G$ is isomorphic to one of the following groups
        \begin{itemize}
        \item A finite abelian group;
        \item The group $(\Z_3)^2\rtimes Q_8$;
        \item $(\Z_p)^r\rtimes \Z_{p^{r-1}}$, for some prime $p$ and $r>0$;
        \item A non-abelian Camina $p-$group, for some prime $p$.
    \end{itemize}
    \end{Thm}

We begin by proving that for each of the groups in Theorem \ref{thm:acclassify}, $T(G)$ is an almost commutative Terwilliger algebra.

    \begin{Prop}\label{thm:abelac}
        Let $G$ be a finite abelian group. Then $T(G)$ is almost commutative.
    \end{Prop}

    \begin{proof}
    Given any $C_i=\{x\}$ and $C_h=\{y\}$ with $x\neq y$ we have that, $C_j=\{x^{-1}y\}$ satisfies $\overline{C_i}\cdot\overline{C_j}=y=\overline{C_h}$. So only $p_{ij}^h\neq 0$ for this specific value of $j$. Hence, by Theorem \ref{thm:tanaka}(3), $T(G)$ is almost commutative. 
    \end{proof}

Next we consider $(\Z_3)^2\rtimes Q_8$. As this is a single group, we can directly compute its conjugacy classes to show the result.

   \begin{Prop}\label{thm: 72,41}
        Let $G=(\Z_3)^2\rtimes Q_8$. Then $T(G)$ is almost commutative.
    \end{Prop}

    \begin{proof}
        A presentation of the group $G=(\Z_3)^2\rtimes Q_8=\langle x,y\rangle\rtimes \langle a,b,c\rangle$ is
        \[\langle a,b,c,x,y\colon a^2=b^2=c, c^2=e, x^3=y^3=e, b^a=bc, x^a=xy^2, x^b=y, x^c=x^2, y^a=x^2y^2, y^b=x^2,y^c=y^2\rangle.\]
        Through direct computation one can find that the six conjugacy classes of $G$ are
        \begin{small}\[C_0=\{e\}, C_1=\{cy^2, cx^2, cx, cx^2y, c, cxy, cx^2y^2, cy, cxy^2\}, C_2=\{xy,x^2y^2,y,xy^2,y^2,x^2,x,x^2y\},\]
        \[C_3=\{aby, abxy^2,abcx^2,ab,abxy,abx^2y^2,abcy^2,abx,abx^2y,abcy, abx^2, abcxy^2, abc, abcxy, abcx^2y^2, aby^2, abcx, abcx^2y\},\]
        \[C_4=\{acx^2,a, axy, ax^2y^2,ax, acy^2, ax^2y, acy, ax^2, acxy^2, ac, acxy, acx^2y^2, ay^2, acx, acx^2y, ay, axy^2\},\]
        \[C_5=\{bxy^2, bcx^2, b, bxy, bx^2y^2, bcy^2, bx, bx^2y, bcy, bx^2, bcxy^2, bc, bcxy, bcx^2y^2, by^2, bcx, bcx^2y, by\}.\]\end{small}
        We shall now compute the products $\overline{C_i}\cdot\overline{C_j}$. Note that $\overline{C_i}\cdot\overline{C_j}=\overline{C_j}\cdot\overline{C_i}$, so it will suffice to compute the products for $i\leq j$. Doing so we find that $\overline{C_0}\cdot \overline{C_j}=\overline{C_j}$ for $0\leq j\leq 5$ and that
        \[\overline{C_1}\cdot\overline{C_1}=9\overline{C_0}+9\overline{C_2},\  \overline{C_1}\cdot\overline{C_2}=8\overline{C_1},\  \overline{C_1}\cdot\overline{C_3}=9\overline{C_3}, \ \overline{C_1}\cdot\overline{C_4}=9\overline{C_4},\  \overline{C_1}\cdot\overline{C_5}=9\overline{C_5},\  \overline{C_2}\cdot\overline{C_2}=8\overline{C_0}+7\overline{C_2},\]  \[\overline{C_2}\cdot\overline{C_3}=8\overline{C_3},\  \overline{C_2}\cdot\overline{C_4}=8\overline{C_4},\  \overline{C_2}\cdot\overline{C_5}=8\overline{C_5},\  \overline{C_3}\cdot\overline{C_3}=18\overline{C_0}+18\overline{C_1}+18\overline{C_2},\  \overline{C_3}\cdot\overline{C_4}=18\overline{C_5},\  \overline{C_3}\cdot\overline{C_5}=18\overline{C_4},\]  \[\overline{C_4}\cdot\overline{C_4}=18\overline{C_0}+18\overline{C_1}+18\overline{C_2},\  \overline{C_4}\cdot\overline{C_5}=18\overline{C_3},\  \overline{C_5}\cdot\overline{C_5}=18\overline{C_0}+18\overline{C_1}+18\overline{C_2}.\]
        Looking at these products we see that given $C_i$ and $C_h$ with $C_i\neq C_h$ there exists a unique $C_j$ such that $p_{ij}^h\neq 0$. Therefore, $T(G)$ is almost commutative by Theorem \ref{thm:tanaka}(3).
    \end{proof}

    \begin{Thm}\label{thm:almost2com}
        Let $G$ be a Camina group in which $G'$ is the union of two conjugacy classes.  Then $T(G)$ is almost commutative.
    \end{Thm}

    \begin{proof}
        Here $G'=C_0\cup C_1$ with $C_1=G'\setminus \{e\}$. We show given distinct $i$ and $h$ there exists a unique $j$ such that $p_{ij}^h\neq 0$. We do so by splitting into cases based on if $C_i\subseteq G'$ or not.

        {\bf Case 1: } $C_i\subseteq G'$. If $C_i=\{e\}$, then $\overline{C_i}\cdot\overline{C_j}=\overline{C_j}$. In this case $p_{ij}^h\neq 0$ if and only if $C_j=C_h$. Hence, there is a unique $j$ such that $p_{ij}^h\neq 0$ in this case. If $C_i=G'\setminus \{e\}$ we consider two subcases.\\
        {\bf Subcase 1.1: }$C_h\subseteq G'$. Then as $C_h\neq C_i=G'\setminus \{e\}$ we have $C_h=\{e\}$. For any $C_j$ not in $G'$ we have by Lemma \ref{lem:pair} that $C_iC_j=C_j$. So $\overline{C_i}\cdot \overline{C_j}=|C_i|\overline{C_j}$. Then in this case as $C_j\neq C_h$ we have $p_{ij}^h=0$. Thus for $p_{ij}^h\neq 0$ either $C_j=\{e\}$ or $C_j=G'\setminus \{e\}$. Since $\overline{C_i}\cdot e=\overline{C_i}\neq \overline{C_h}$ we have $p_{ij}^h=0$ in this case. The only class we have left to consider is $C_j=G'\setminus \{e\}$. In this case $\overline{C_i}\cdot\overline{C_j}=(\overline{G'}-e)(\overline{G'}-e)=|G'|\overline{G'}-2\overline{G'}+e=(|G'|-2)\overline{G'}+e$. Thus, $p_{ij}^h\neq 0$, and there is exactly one $j$ such that $p_{ij}^h\neq 0$.\\
        {\bf Subcase 1.2: }$C_h\subseteq G\setminus G'$. Then for any $C_j\subseteq G'$ we have $C_iC_j\subseteq G'$ as $C_i,C_j\subseteq G'$. Then $p_{ij}^h=0$ in this case. For any $C_j\subseteq G\setminus G'$ we have by Lemma \ref{lem:pair} that $C_iC_j=C_j$. Then for $p_{ij}^h\neq 0$ we must have $j=h$. That is, $p_{ij}^h\neq 0$ if and only if $j=h$ in this case.

        {\bf Case 2: }$C_i\subseteq G\setminus G'$. Let $g\in C_i$. Then $C_i=gG'$. Now consider two subcases.\\
        {\bf Subcase 2.1: }$C_h\subseteq G'$. If $C_j\subseteq G'$, then $C_iC_j=(gG')C_j\subseteq gG'$. As $gG'$ is disjoint from $G'$, $p_{ij}^h=0$ for all such $j$. Now take $C_j\subseteq G\setminus G'$. So $C_j=kG'$ for some $k\in G\setminus G'$. Then $C_iC_j=(gG')(kG')=(gk)G'$. This means that all elements of $\overline{C_i}\cdot\overline{C_j}$ are in $(gk)G'$. As $C_h\subseteq G'$ we have $p_{ij}^h\neq 0$ if and only if $gkG'=G'$. In this case $gk\in G'$, so $kG'=g^{-1}G'$. As inverses in $G/G'$ are unique there is a unique $C_j$ such that $C_iC_j=G'$ and $p_{ij}^h\neq 0$.\\
        {\bf Subcase 2.2: }$C_h\subseteq G\setminus G'$. Then $C_h=xG'$ for $x\in G\setminus G'$, with $gG'\neq xG'$. If $C_j\subseteq G'$ then $C_iC_j\subseteq gG'$ and $p_{ij}^h=0$. Now if $C_j$ is outside of $G'$, $C_j=kG'$ and $C_iC_j=gkG'$. Hence, all the terms in $\overline{C_i}\cdot\overline{C_j}$ are in $gkG'$. Then $p_{ij}^h\neq 0$ if and only if $gkG'=xG'$ if and only if $C_j=kG'=g^{-1}xG'$. Therefore, there is a unique $j$ such that $p_{ij}^h\neq 0$.

        Hence, $T(G)$ is almost commutative by Theorem \ref{thm:tanaka}(3).
    \end{proof}

    With this result and Theorem \ref{thm:twoclass} we immediately have the following result.

\begin{Cor}\label{cor:frobac}
        Let $G$ be the Frobenius group $Z_{p}^r\rtimes Z_{p^r-1}$ for some prime $p$ and $r>0$. Then $T(G)$ is almost commutative. \qed
    \end{Cor}

     We now turn our attention to Camina $p-$groups.

    \begin{Thm}\label{thm:Caminapac}
        Let $G$ be a Camina $p-$group. Then $T(G)$ is almost commutative.
    \end{Thm}

    \begin{proof}
        Let $\gamma_2(G)=[G,G]$ and $\gamma_3(G)=[\gamma_2(G),G]$. Let $C_i=x^G$ and $C_h=y^G$ with $i\neq h$. We consider nine cases coming from Proposition \ref{lem:pcamclass}.
        
        {\bf Case 1: }$x,y\in \gamma_3(G)$. Then by Proposition \ref{lem:pcamclass}, $x^G=\{x\}$, $y^G=\{y\}$, and $x^{-1}y\in \gamma_3(G)$. Thus, by Proposition \ref{lem:pcamclass}, $(x^{-1}y)^G=\{x^{-1}y\}$. Set $C_j=(x^{-1}y)^G$. Then $C_iC_j=\{y\}=C_h$. So $p_{ij}^h\neq 0$. Suppose $C_k=z^G$ is another conjugacy classes of $G$ such that $p_{ik}^h\neq 0$. Then $y\in C_iC_k$. So there exists some $gzg^{-1}\in z^G$ such that $y=xgzg^{-1}$. Then $x^{-1}y=gzg^{-1}\in C_j$. So $C_k=C_j$ and $C_j$ is the unique conjugacy class such that $p_{ij}^h\neq 0$ in this case.
        
        {\bf Case 2: }$x\in \gamma_2(G)\setminus \gamma_3(G)$ and $y\in \gamma_3(G)$. Then by Proposition \ref{lem:pcamclass}, $C_i=x\gamma_3(G)$ and $C_h=\{y\}$. As $x\in \gamma_2(G)\setminus \gamma_3(G)$, we have $x^{-1}\in \gamma_2(G)\setminus \gamma_3(G)$ as well. Then by Proposition \ref{lem:pcamclass}, we have $(x^{-1})^G=x^{-1}\gamma_3(G)$. Set $C_j=(x^{-1})^G$. We have $C_iC_j=x\gamma_3(G)x^{-1}\gamma_3(G)=\gamma_3(G)$. As $y\in \gamma_3(G)$, $p_{ij}^h\neq 0$. Suppose $C_k=z^G$ with $p_{ik}^h\neq 0$. Then $y\in C_iC_k$. If $z\in \gamma_3(G)$, then by Proposition \ref{lem:pcamclass}, $z^G=\{z\}$. So we have $C_iC_k=x\gamma_3(G)z=x\gamma_3(G)$ since $z\in \gamma_3(G)$. Note $y\not\in x\gamma_3(G)$, so we have a contradiction. Thus we must have $z\not\in \gamma_3(G)$. If $z\in G\setminus \gamma_2(G)$, then by Proposition \ref{lem:pcamclass} $z^G=z\gamma_2(G)$. Then $C_iC_k=x\gamma_3(G)z\gamma_2(G)=z\gamma_2(G)$ as $x\gamma_3(G)\subseteq \gamma_2(G)$. Since $y\in \gamma_3(G)\subseteq \gamma_2(G)$, we have that $y\not\in z\gamma_2(G)$. We then have a contradiction. Hence, $z\not\in G\setminus \gamma_2(G)$. Therefore $z\in \gamma_2(G)\setminus \gamma_3(G)$. Then by Proposition \ref{lem:pcamclass}, $z^G=z\gamma_3(G)$. Then $C_iC_k=xz\gamma_3(G)$. We have $y\in xz\gamma_3(G)$ if and only if $xz\gamma_3(G)=\gamma_3(G)$. Then $xz\in \gamma_3(G)$. So $z\in x^{-1}\gamma_3(G)$ and $z^G=x^{-1}\gamma_3(G)=C_j$. Therefore, $C_k=C_j$. Hence, $C_j$ is the unique class of $G$ such that $p_{ij}^h\neq 0$.

        {\bf Case 3: }$x\in G\setminus \gamma_2(G)$ and $y\in \gamma_3(G)$. By Proposition \ref{lem:pcamclass}, $C_i=x\gamma_2(G)$ and $C_h=\{y\}$. As $x\in G\setminus \gamma_2(G)$, we have $x^{-1}\in G\setminus \gamma_2(G)$ as well. Then by Proposition \ref{lem:pcamclass}, we have $(x^{-1})^G=x^{-1}\gamma_2(G)$. Set $C_j=(x^{-1})^G$. We have $C_iC_j=x\gamma_2(G)x^{-1}\gamma_2(G)=\gamma_2(G)$. As $y\in \gamma_3(G)\subseteq \gamma_2(G)$, we have that $p_{ij}^h\neq 0$. Suppose $C_k=z^G$ such that $p_{ik}^h\neq 0$. Then $y\in C_iC_k$. If $z\in \gamma_3(G)$, then by Proposition \ref{lem:pcamclass}, $z^G=\{z\}$. So $C_iC_k=x\gamma_2(G)z=x\gamma_2(G)$ since $z\in \gamma_3(G)\subseteq \gamma_2(G)$. Note as $y\in\gamma_3(G)$, that $y\not\in x\gamma_2(G)$. We then have a contradiction, so $z\not\in \gamma_3(G)$. If $z\in \gamma_2(G)\setminus \gamma_3(G)$, then by Proposition \ref{lem:pcamclass}, $z^G=z\gamma_3(G)$. Then $C_iC_k=x\gamma_2(G)z\gamma_3(G)=x\gamma_2(G)$ as $z\gamma_3(G)\subseteq \gamma_2(G)$. Since $y\in \gamma_3(G)\subseteq \gamma_2(G)$, we have $y\not\in z\gamma_2(G)$. We then have a contradiction, so $z\not\in \gamma_2(G)\setminus \gamma_3(G)$. Thus, $z\in G\setminus \gamma_2(G)$. Then by Proposition \ref{lem:pcamclass}, $z^G=z\gamma_2(G)$. Then $C_iC_k=xz\gamma_2(G)$. We have $y\in xz\gamma_2(G)$ if and only if $xz\gamma_2(G)=\gamma_2(G)$ as $y\in \gamma_3(G)\subseteq \gamma_2(G)$. Then $xz\in \gamma_2(G)$. So $z\in x^{-1}\gamma_2(G)$. Then $z^G=x^{-1}\gamma_2(G)=C_j$, so $C_k=C_j$. Hence, $C_j$ is the unique class of $G$ such that $p_{ij}^h\neq 0$.

        {\bf Case 4: }$x\in \gamma_3(G)$ and $y\in \gamma_2(G)\setminus \gamma_3(G)$. Then by Proposition \ref{lem:pcamclass}, $C_i=\{x\}$ and $C_h=y\gamma_3(G)$. Take $C_j=y\gamma_3(G)$. Then $C_iC_j=x\cdot y\gamma_3(G)=y\gamma_3(G)$ as $x\in \gamma_3(G)$. So $p_{ij}^h\neq 0$. Suppose $C_k=z^G$ such that $p_{ik}^h\neq 0$. Then $y\gamma_3(G)\subseteq C_iC_k$. If $z\in \gamma_3(G)$, then by Proposition \ref{lem:pcamclass}, $z^G=\{z\}$. Then $C_iC_k=xz\in \gamma_3(G)$. As $y\gamma_3(G)\not\subseteq \gamma_3(G)$, we have a contradiction. So $z\not\in \gamma_3(G)$. If $z\in G\setminus \gamma_2(G)$, then by Proposition \ref{lem:pcamclass}, $z^G=z\gamma_2(G)$. Then $C_iC_k=x\cdot z\gamma_2(G)=z\gamma_2(G)$ as $x\in \gamma_3(G)\subseteq \gamma_2(G)$. We have $y\gamma_3(G)\subseteq \gamma_2(G)$, so $y\gamma_3(G)\not\subseteq z\gamma_2(G)$ and we have a contradiction. Then $z\not\in G\setminus \gamma_2(G)$. We then must have $z\in \gamma_2(G)\setminus \gamma_3(G)$. Then by Proposition \ref{lem:pcamclass}, $z^G=z\gamma_3(G)$ and $C_iC_k=x\cdot z\gamma_3(G)=z\gamma_3(G)$. This contains $y\gamma_3(G)$ if and only if $y\gamma_3(G)=z\gamma_3(G)$. So $C_k=C_j$. Therefore, $C_j$ is the unique class of $G$ such that $p_{ij}^h\neq 0$.

        {\bf Case 5: } $x,y\in \gamma_2(G)\setminus \gamma_3(G)$. Then by Proposition \ref{lem:pcamclass}, $C_i=x\gamma_3(G)$ and $C_h=y\gamma_3(G)$ with $x\neq y$ since $C_i\neq C_h$. Pick $C_j=(x^{-1}y)^G$. Note $x^{-1}y\in \gamma_2(G)$ as $x,y\in \gamma_2(G)\setminus \gamma_3(G)$. If $x^{-1}y\in \gamma_3(G)$, then $y\in x\gamma_3(G)$ which contradicts $C_i\neq C_h$. Thus, $x^{-1}y\in \gamma_2(G)\setminus \gamma_3(G)$. Then by Proposition \ref{lem:pcamclass}, $C_j=x^{-1}y\gamma_3(G)$. So we have $C_iC_j=x\gamma_3(G)(x^{-1}y\gamma_3(G))=y\gamma_3(G)$. So $p_{ij}^h\neq 0$. Suppose $C_k=z^G$ such that $p_{ik}^h\neq 0$. Then $y\gamma_3(G)\subseteq C_iC_k$. If $z\in \gamma_3(G)$, then $z^G=\{z\}$ by Proposition \ref{lem:pcamclass}. Then $C_iC_k=x\gamma_3(G)\cdot z=x\gamma_3(G)$ as $z\in \gamma_3(G)$. This does not contain $y\gamma_3(G)$ as $C_i\neq C_h$, so we have a contradiction. This implies $z\not\in \gamma_3(G)$. If $z\in G\setminus \gamma_2(G)$, then $z^G=z\gamma_2(G)$ by Proposition \ref{lem:pcamclass}. Then $C_iC_k=x\gamma_3(G)z\gamma_2(G)=z\gamma_2(G)$ as $x\gamma_3(G)\subseteq \gamma_2(G)$. As $y\gamma_3(G)\subseteq \gamma_2(G)$, $y\gamma_3(G)\not\subseteq z\gamma_2(G)$. This gives contradiction, so $z\not\in G\setminus \gamma_2(G)$. Then $z\in \gamma_2(G)\setminus \gamma_3(G)$, so by Proposition \ref{lem:pcamclass}, $z^G=z\gamma_3(G)$. Then $C_iC_k=xz\gamma_3(G)$. This contains $y\gamma_3(G)$ if and only if $xz\gamma_3(G)=y\gamma_3(G)$. Then $z\gamma_3(G)=x^{-1}y\gamma_3(G)=C_j$. So $C_k=C_j$. Therefore, $C_j$ is the unique class of $G$ such that $p_{ij}^h\neq 0$. 

        {\bf Case 6: }$x\in G\setminus \gamma_2(G)$ and $y\in \gamma_2(G)\setminus \gamma_3(G)$. Then by Proposition \ref{lem:pcamclass}, $C_i=x\gamma_2(G)$ and $C_h=y\gamma_3(G)$. Take $C_j=(x^{-1})^G$. As $x\in G\setminus \gamma_2(G)$ we have $x^{-1}\in G\setminus \gamma_2(G)$. Then by Proposition \ref{lem:pcamclass}, $C_j=x^{-1}\gamma_2(G)$. Thus $C_iC_j=xx^{-1}\gamma_2(G)=\gamma_2(G)$. As $y\gamma_3(G)\subseteq \gamma_2(G)$, $p_{ij}^h\neq 0$. Suppose $C_k=z^G$ such that $p_{ik}^h\neq 0$. Then $y\gamma_2(G)\subseteq C_iC_k$. If $z\in \gamma_3(G)$, then by Proposition \ref{lem:pcamclass}, $z^G=\{z\}$. Then $C_iC_k=x\gamma_2(G)\cdot z=x\gamma_2(G)$ as $z\in \gamma_3(G)\subseteq \gamma_2(G)$. As $y\gamma_3(G)\not\subseteq x\gamma_2(G)$ since $y\gamma_3(G)\subseteq \gamma_2(G)$ we have a contradiction. So $z\not\in \gamma_3(G)$. If $z\in \gamma_2(G)\setminus \gamma_3(G)$, then by Proposition \ref{lem:pcamclass}, we have $z^G=z\gamma_3(G)$. Then $C_iC_k=x\gamma_2(G)\cdot z\gamma_3(G)=x\gamma_2(G)$ as $z\gamma_3(G)\subseteq \gamma_2(G)$. As $y\gamma_3(G)\not\subseteq x\gamma_2(G)$ since $y\gamma_3(G)\subseteq \gamma_2(G)$ we have $z\not\in \gamma_2(G)\setminus \gamma_3(G)$. Then $z\in G\setminus \gamma_2(G)$. Then by Proposition \ref{lem:pcamclass}, $z^G=z\gamma_2(G)$. We then have $C_iC_k=xz\gamma_2(G)$. This contains $y\gamma_3(G)$ if and only if $xz\gamma_2(G)=\gamma_2(G)$. Then $z\gamma_2(G)=x^{-1}\gamma_2(G)$. So $C_k=C_j$. Therefore, $C_j$ is the unique class of $G$ such that $p_{ij}^h\neq 0$.

        {\bf Case 7: }$x\in \gamma_3(G)$ and $y\in G\setminus \gamma_2(G)$. The proof is similar to Case 4.

        {\bf Case 8: }$x\in \gamma_2(G)\setminus \gamma_3(G)$ and $y\in G\setminus \gamma_2(G)$. The proof is similar to Case 5.

        {\bf Case 9: }$x,y\in G\setminus \gamma_2(G)$. The proof is similar to Case 6.

        Therefore, $T(G)$ is almost commutative by Theorem \ref{thm:tanaka}(3).
    \end{proof}

    At this point, we have proved that for each of the four types of groups $G$ in Theorem \ref{thm:acclassify}, $T(G)$ is almost commutative. For the converse we need:

    \begin{Lem}\label{thm:acimplies}
        Suppose that $T(G)$ is almost commutative. Then for $x,y\in G$ with $x^G\neq (y^{-1})^G$, we have $x^Gy^G=(xy)^G$.
    \end{Lem}

    \begin{proof}
        Let $x,y\in G$ such that $x^G\neq (y^{-1})^G$. Let $x^G=C_i$ and $y^G=C_j$. Then $C_j\neq C_{i'}$. Since $C_j\neq C_{i'}$, and $T(G)$ is almost commutative, there exists a unique $h$ with $p_{i'h}^j\neq 0$, so $C_j\subseteq C_{i'}C_h$. Then for $z\in C_j$, there are $a\in C_{i'}$ and $b\in C_h$ such that $z=ab$. Then $b=a^{-1}z$, so $b\in C_iC_j$. Thus, $C_h\subseteq C_iC_j$.

        Now suppose that $C_m\subseteq C_iC_j$. Then for $a\in C_m$, there exists a $b\in C_i$ and $c\in C_j$ such that $a=bc$. Then $c=b^{-1}a$, so $c\in C_{i'}C_m$. Thus $C_j\subseteq C_{i'}C_m$. This implies $p_{i'm}^j\neq 0$. By the uniqueness of $p_{i'h}^j$, we have $C_m=C_h$. So $C_h$ is the only class in $C_iC_j$. We know that $\overline{C_i}\cdot\overline{C_j}=\sum_{\ell} p_{ij}^\ell \overline{C_\ell}$. For all $p_{ij}^\ell\neq 0$ we have $C_\ell\subseteq C_iC_j$. Since $C_h$ is the unique class in $C_iC_j$ we must have $C_iC_j=C_h$.

        Let us consider what $C_h$ could be. We know that $x\in C_i$ and $y\in C_j$ so $xy\in C_h$ as $C_iC_j=C_h$. Therefore, $C_h=(xy)^G$. Thus $x^Gy^G=(xy)^G$. 
    \end{proof}

    The following result now shows which groups $G$ satisfy the conclusion of Lemma \ref{thm:acimplies}.

    \begin{Thm}[Dade and Yadav, \cite{Dade}]\label{thm:groupequiv} A finite group $G$ satisfies the property that for all $x,y\in G$ such that $x^G\neq (y^{-1})^G$, $x^Gy^G=(xy)^G$ if and only if $G$ is isomorphic to exactly one of the groups in the following list:
    \begin{itemize}
        \item A finite abelian group.
        \item A non-abelian Camina $p-$group, for some prime $p$.
        \item $(\Z_p)^r\rtimes \Z_{p^{r-1}}$ for some prime $p$ and $r>0$.
        \item The group $(\Z_3)^2\rtimes Q_8$.
    \end{itemize}
    \end{Thm}

    Combining Theorem \ref{thm:groupequiv} and Lemma \ref{thm:acimplies} we see for a group $G$, if $T(G)$ is almost commutative, then $G$ must be one of the four types of groups in Theorem \ref{thm:acclassify}. This completes the proof of Theorem \ref{thm:acclassify}. \qed 

\section{Wedderburn Decomposition for Almost Commutative Group Association Schemes}

Now that we have found exactly which groups have a group association scheme that produces an almost commutative Terwilliger algebra, we are going to find the Wedderburn decomposition of each of these Terwilliger algebras.

For finite abelian groups, we have:

\begin{Prop}[Bannai and Munemasa, \cite{Bannaiarticle}] Let $G$ be a finite abelian group. Then $\dim T(G)=|G|^2$.
\end{Prop}

We note that the primary component of the Wedderburn decomposition has dimension equal to the number of conjugacy classes squared. For an abelian group this is $|G|^2$. Hence, the Wedderburn decomposition of $T(G)$ is just $V$, where $V$ is the primary component.

For $(\Z_3)^2\rtimes Q_8$ we have:

\begin{Prop}\label{cor:72,41 dim}
        Let $G=(\Z_3)^2\rtimes Q_8$. Then $\dim T(G)=44$.
    \end{Prop}

    \begin{proof}
        By Theorem \ref{thm: 72,41}, $T(G)$ is almost commutative. Then by Theorem \ref{thm:tanaka}, $T(G)$ is triply regular. Therefore, $T(G)=T_0(G)$ and $\dim T(G)=|\{(i,j,k)\colon p_{ij}^k\neq 0\}|$. Within the proof of Theorem \ref{thm: 72,41} we computed every product of the form $\overline{C_i}\cdot\overline{C_j}$ for $i\leq j$. Then in computing the number of $(i,j,k)$ such that $p_{ij}^k\neq 0$ we count the number of $p_{ij}^k\neq 0$ that we got in the products in the proof of Theorem \ref{thm: 72,41}. In this counting, for each $i\neq j$ we get two nonzero triples, namely $(i,j,k)$ and $(j,i,k)$ where $p_{ij}^k\neq 0$. Doing this we find $|\{(i,j,k)\colon p_{ij}^k\neq 0\}|=44$, so $\dim T(G)=44$. 
        \end{proof}

        We next determine the dimension of $Z(T(G))$, the center of $T(G)$. This is equal to the number of Wedderburn components $T(G)$ has. We use the method of \cite{GroupAssoc1}.

        \begin{Prop}\label{prop:72,41 center}
            Let $G=(\Z_3)^2\rtimes Q_8$. Then $\dim Z(T(G))=9$.
        \end{Prop}

        \begin{proof}
            Since $T(G)$ contains all the $E_i^*$ matrices, $Z(T(G))$ consists of block diagonal matrices and
            \[Z(T(G))\subseteq \bigoplus_{i=0}^d Z(E_i^*T(G)E_i^*).\]
            To find $\dim Z(T(G))$, we explicitly determine a basis for $Z(T(G))$. To do so, we first determine a basis for $Z(E_i^*T(G)E_i^*)$ for each $i=0,1,\cdots, d$. Let $\{b_j\}$ be the union of these bases. Thus if $y\in Z(T(G))$, $y=\sum c_j b_j$, for some $c_j\in \Z$. The system of linear equations $\{x_iy=yx_i\}$, ranging over all elements $x_i\in \{E_i^*A_jE_k^*\colon p_{ij}^k\neq 0\}$, the basis of $T(G)=T_0(G)$, can then be solved to yield the scalars $c_j$ and the required basis for $Z(T(G))$. As we are looking at $G=(\Z_3)^2\rtimes Q_8$ this calculation can be done using Magma \cite{Magma}, to find $\dim Z(T(G))=9$.
        \end{proof}

        \begin{Thm}\label{cor: 72,41 decomp}
            Let $G=(\Z_3)^2\rtimes Q_8$. Then the Wedderburn decomposition of $T(G)$ is
            \[T(G)\cong V\oplus Z_1\oplus Z_2 \oplus Z_3 \oplus \cdots \oplus Z_8 \]
            where each $Z_i$, $1\leq i\leq 8$ is a $1-$dimensional ideal of $T(G)$ and $V$ is the primary component. 
        \end{Thm}

        \begin{proof}
            Recall that $\dim Z(T(G))$ equals the number of irreducible $2-$sided ideals of $T(G)$. Then by Proposition \ref{prop:72,41 center}, $T(G)$ has $9$ irreducible $2-$sided ideals. We know that $\dim V$ is the square of the number of classes of $G$. Note $G$ has $6$ classes, so $\dim V=36$. From Corollary \ref{cor:72,41 dim}, $\dim T(G)=44$, so the sum of the dimensions of the Wedderburn components of $T(G)$ outside of $V$ is $8$. As $T(G)$ has a total of $9$ irreducible two-sided ideals, $8$ of which are not $V$, each of the other $8$ is $1-$dimensional.
        \end{proof}

        Using Magma, one can directly compute the primitive idempotents corresponding to each of these nine Wedderburn components of $T((\Z_3)^2\rtimes Q_8)$. 
        
        For Frobenius groups $(\Z_p)^n\rtimes \Z_{p^n-1}$, we have:

         \begin{Thm}\label{cor:frobdim}
            Let $G=\Z_{p}^n\rtimes \Z_{p^n-1}$ for some prime $p$ and $n\geq 1$. Then $\dim T(G)=p^{2n}+p^n-1$.
        \end{Thm}

        \begin{proof}
         By Corollary \ref{cor:frobac}, $T(G)$ is almost commutative. Then by Theorem \ref{thm:tanaka}, $T(G)$ is triply regular. Therefore, $T(G)=T_0(G)$ and $\dim T(G)=|\{(i,j,k)\colon p_{ij}^k\neq 0\}|$. 
         
         Since $G$ is a Camina group, its classes outside of $G'$ are all nontrivial cosets of $G'$. By Theorem \ref{thm:twoclass}, the classes of $G$ inside $G'$ are $\{e\}$ and $G'\setminus \{e\}$. Note $G'\cong \Z_{p}^n$ and $G/G'\cong \Z_{p^n-1}$. Then there are $p^n-2$ classes of $G$ outside of $G'$. So the number of classes of $G$ is $2+(p^n-2)=p^n$. We count the number of triples by considering cases based the types of classes. Let $C_0=\{e\}$ and $C_1=G'\setminus \{e\}$.

        {\bf Case 1: }$C_i=C_0$. Now $\overline{C_0}\cdot\overline{C_j}=\overline{C_j}$ for all $j$, thus there are $p^n$ triples $(0,j,j)$ (or $(j,0,j)$) such that $p_{0j}^j\neq 0$ (or $p_{j0}^j\neq 0$), giving a total of $2p^n$ triples that involve the class $C_i=\{e\}$, $p_{ij}^h\neq 0$. This double counts the triple $(0,0,0)$. Hence, there are $2p^n-1$ nonzero triples that involve $C_0=\{e\}$.

        {\bf Case 2: }$C_i=C_1$. We consider $\overline{C_1}\cdot\overline{C_j}$, where $j>1$. Then $\overline{C_1}\cdot\overline{C_j}=\overline{(G'-e)}\cdot\overline{gG'}=|G'|\overline{gG'}-\overline{gG'}=(|G'|-1)\overline{gG'}=(p^n-1)\overline{C_j}$. So $(1,j,j)$ is a triple with $p_{1j}^j\neq 0$. Similarly, $(j,1,j)$ has $p_{j1}^j\neq 0$. There are $p^n-2$ possible choices for $C_j$, so we have found $2(p^n-2)$ triples. Now $\overline{C_1}\cdot\overline{C_1}=\overline{G'\setminus \{e\}}\cdot \overline{G'\setminus \{e\}}=|G'|\overline{G'}-2\overline{G'}+\{e\}=(p^n-2)\overline{G'\setminus \{e\}}+(p^n-1)\overline{\{e\}}$. Then the triples $(1,1,0)$ and $(1,1,1)$ also have nonzero $p_{11}^0$ and $p_{11}^1$. Thus, when $C_1$ is a class in the product we get $2p^n-4+2=2p^n-2$ nonzero triples.

        {\bf Case 3: } $C_i=gG'$ where $g\not\in G'$. We may consider $C_j=hG'$ as we have already counted all the classes that are not nontrivial cosets of $G'$ in $G$. For all $C_j\neq g^{-1}G'$, we have $C_iC_j=ghG'\neq G'$. This implies $\overline{C_i}\cdot\overline{C_j}=|G'|\overline{ghG'}=|G'|\overline{C_k}$ for some conjugacy class $C_k$. Thus, $(i,j,k)$ is the only triple with our choice of $i,j$ such that $p_{ij}^k\neq 0$. With $C_i$ fixed, we have a total of $p^n-2-1$ choices for $C_j$ such that $C_j\neq g^{-1}G'$. As each choice of $C_j$ gives a unique $C_k$ such that $p_{ij}^k\neq 0$ we get a total of $p^n-3$ triples $(i,j,k)$ in this case. If $C_j=g^{-1}G'$, then $\overline{C_i}\cdot\overline{C_j}=\overline{gG'}\cdot\overline{g^{-1}G'}=|G'|\overline{G'}=|G'|\overline{C_0}+|G'|\overline{C_1}$. Thus, the triples $(i,j,0)$ and $(i,j,1)$ have nonzero $p_{ij}^0$ and $p_{ij}^1$ in this case. We then have found a total of $p^n-1$ different $(i,j,k)$ such that $p_{ij}^k\neq 0$ for our fixed $C_i$. We have a total of $p^n-2$ choices for $C_i$, hence there are a total of $(p^n-1)(p^n-2)$ triples $(i,j,k)$ such that $p_{ij}^k\neq 0$. This gives
        \[\dim(T(G))=2p^n-1+2p^n-2+(p^n-2)(p^n-1)=p^{2n}+p^n-1. \qedhere\]
        \end{proof}

For the nonprimary Wedderburn components of $T(G)$, $G=(\Z_p)^n\rtimes \Z_{p^n-1}$, we do the following. Order the elements of $G$ so all the elements of $C_0$ are first, then all those in $C_1$ and so on. Then every element of $T(G)$ can be thought of as a block matrix where the $C_i,C_j$ block of a matrix $B\in T(G)$ refers to those entries $a,b$ of $B$ with $a\in C_i$ and $b\in C_j$. We shall assume this ordering from this point on. Then the primary component $V$ has a basis $\{v_{ij}\colon 0\leq i,j\leq d\}$ where $v_{ij}$ has all $1$'s in the $C_i,C_j$ block and is $0$ outside this block. The primitive idempotent corresponding to $V$ in this case is block diagonal with the $C_i,C_i$ block having $\frac{1}{|C_i|}$ in every entry, for each $i$.

\begin{Prop}\label{Prop:frob1dim}
    Let $G=Z_p^{n}\rtimes Z_{p^n-1}$ for some prime $p$ and $n\geq 1$. Suppose $C_0=\{e\},C_1,\cdots, C_{p^n-1}$ are the classes of $G$. Define $B_i$ to be the $|G|\times |G|$ matrix indexed by $G$ which is $0$ outside the $C_i,C_i$ block and has $C_i,C_i$ block
    \[d_{ii}=\frac{-1}{|C_i|-1}J_{|C_i|}+\left(1-\frac{-1}{|C_i|-1}\right)I_{|C_i|}=\begin{pmatrix}
        1 & \frac{-1}{|C_i|-1} & \frac{-1}{|C_i|-1} & \cdots & \frac{-1}{|C_i|-1} \\ \frac{-1}{|C_i|-1} & 1 & \frac{-1}{|C_i|-1} & \cdots & \frac{-1}{|C_i|-1} \\
        \vdots & & \ddots & & \vdots \\
        \vdots & & & \ddots & \vdots \\
        \frac{-1}{|C_i|-1} & \frac{-1}{|C_i|-1} & \cdots & \frac{-1}{|C_i|-1} & 1
    \end{pmatrix}\]
    where $J_{|C_i|}$ is the $|C_i|\times |C_i|$ all $1$'s matrix and $I_{|C_i|}$ is the $|C_i|\times |C_i|$ identity matrix. Then for $1\leq i\leq p^n-1$, $W_i=\Span(B_i)$ is a $2-$sided ideal of $T(G)$.
\end{Prop}

\begin{proof}
    From the proof of Corollary \ref{cor:frobdim} the conjugacy classes of $G$ are $\{e\},G'\setminus \{e\}$, and the nontrivial cosets of $G'$.
    We start by proving that for all $1\leq i\leq p^n-1$, $B_i\in T(G)$. Recall that $\sum_{k=0}^{p^n-1} A_k=J$, the all $1$'s matrix, and that $A_0=I_{|\Omega|}$. Consider for any $1\leq i\leq p^n-1$, 
    \[E_i^*\left(A_0-\frac{1}{|C_i|-1}\sum_{k=1}^{p^n-1} A_k\right)E_i^*=E_i^*\left(I-\frac{1}{|C_i|-1}(J-I)\right)E_i^*.\]
    This matrix is $0$ outside of the $C_i,C_i$ block and the $C_i,C_i$ block is $d_{ii}$. Hence, for $1\leq i\leq p^n-1$, $B_i=E_i^*\left(A_0-\frac{1}{|C_i|-1}\sum_{k=1}^{p^n-1} A_k\right)E_i^*\in T(G)$ for all $i$.

    We now show that $W_i$ is closed under multiplication by any element of $T(G)$. To do so, we show it is closed under multiplication by the generators of $T(G)$. For the $E_i^*$ matrices, $E_i^*E_j^*=E_i^*$ if $i=j$ and $0$ otherwise. Then, as $B_i=E_i^*\left(A_0-\frac{1}{|C_i|-1}\sum_{k=1}^{p^n-1} A_k\right)E_i^*$, we have
    \[E_j^*B_i=\left\{\begin{array}{cc}
        B_i & \text{ if }i=j  \\
        0 &  \text{ otherwise }
    \end{array}\right.,\ \ \ \ \ B_iE_j^*=\left\{\begin{array}{cc}
        B_i & \text{ if }i=j  \\
        0 & \text{ otherwise } 
    \end{array}\right..\]
    So for any $i$, $W_i$ is closed under multiplication by the $E_j^*$ matrices.

    We now consider $A_kB_i$. Using block notation and letting $a_{t,i}$ be the $C_t,C_i$ block of $A_k$, with $*$ symbolizing an unspecified value, we have

\begin{equation}\label{eq:blockmultiply3}
A_kB_i=\begin{pmatrix} * & \cdots & a_{0,i} & * & \cdots\\ * & \cdots & a_{1,i} & * &\cdots \\
\vdots & \vdots & \vdots & \vdots & \vdots\\
* & \cdots & a_{t,i} & * &\cdots \\
\vdots & \vdots & \vdots & \vdots & \vdots\\
* & \cdots & a_{n,i} & * & \cdots
\end{pmatrix}\begin{pmatrix}  0 & \cdots & 0 & 0 & \cdots\\ 0 & \cdots & 0 & 0 &\cdots \\
\vdots & \vdots & \vdots & \vdots & \vdots\\
0 & \cdots & d_{ii} & 0 &\cdots \\
\vdots & \vdots & \vdots & \vdots & \vdots\\
0 & \cdots & 0 & 0 & \cdots\end{pmatrix}= \begin{pmatrix} 0 & \cdots & a_{0,i}d_{ii} & 0 & \cdots\\ 0 & \cdots & a_{1,i}d_{ii} & 0 &\cdots \\
\vdots & \vdots & \vdots & \vdots & \vdots\\
0 & \cdots & a_{t,i}d_{ii} & 0 &\cdots \\
\vdots & \vdots & \vdots & \vdots & \vdots\\
0 & \cdots & a_{n,i}d_{ii} & 0 & \cdots \end{pmatrix}
\end{equation}
where $d_{ii}$ is as defined above.
    We consider the following cases arising from choices of blocks of $A_k.$

    {\bf Case 1: }The $C_0,C_i$ block $a_{0,i}$ of $A_k$ is a $1\times |C_i|$ matrix. If the entry $(e,x)$ is $1$, then $x\in C_k$. Since we are considering the $C_0,C_i$ block we have $x\in C_i$, so $C_i=C_k$ in this case. Then every entry in the block is nonzero. So the $C_0,C_i$ block of $A_k$ is either all $0$'s or all $1$'s. Notice if $a_{0,i}$ is all $0$'s we have $a_{0,i}d_{ii}=0$. If $a_{0,i}$ is all $1$'s then $a_{0,i}d_{ii}$ is just the column sum of each column of $d_{ii}$, which is $0$. So $a_{0,i}d_{ii}=0$ in this case as well. Thus, $a_{0,i}d_{ii}=0$ in either case.

    {\bf Case 2: }Consider the $C_t,C_i$ block, $a_{t,i}$, of $A_k$ when $C_t=G'\setminus \{e\}$. We consider subcases based on if $C_i=C_t$ or not.\\
    {\bf Subcase 2.1: } $C_i=C_t$. If the $(x,y)$ entry of $A_k$ is nonzero, then $yx^{-1}\in C_k$. As $x\in C_t$ and $y\in C_i=C_t$ we have $x,y\in G'$. Then $yx^{-1}\in G'$, and as $yx^{-1}\in C_k$ either $C_k=\{e\}$ or $C_k=G'\setminus \{e\}$. In the first case $A_k=I_n$, so $a_{t,i}d_{ii}=d_{ii}$. In the second case we have for all $x\neq y$, $yx^{-1}\in G'\setminus \{e\}$. So the $(x,y)$ entry of $A_k$ is nonzero for all $x\neq y$. Then $a_{t,i}=J_{|C_i|}-I_{|C_i|}$, and $a_{t,i}d_{ii}=J_{|C_i|}d_{ii}-d_{ii}$. Note $J_{|C_i|}d_{ii}=0$, so $a_{t,i}d_{ii}=-d_{ii}$ in this case. Thus, if $a_{t,i}$ is nonzero in some entry then either $a_{t,i}d_{ii}=d_{ii}$ or $a_{t,i}d_{ii}=-d_{ii}$. Note if $a_{t,i}=0$, then $a_{t,i}d_{ii}=0$.\\
    {\bf Subcase 2.2: } $C_i\neq C_t$. Consider the $(x,y)$ entry of $a_{t,i}$. We have $x\in C_t$ and $y\in C_i$. As $C_t=G'\setminus \{e\}$, $x\in G'$. Since $C_i\neq C_t$, $C_i=zG'$ for some $z\in G\setminus G'$. Then $yx^{-1}\in zG'$. Thus, the $(x,y)$ entry of $a_{t,i}$ is nonzero if and only if $C_k=zG'=C_i$. In this case $yx^{-1}\in zG'$ for all $x\in C_t=G'\setminus \{e\}$ and $y\in zG'$. So every entry of $a_{t,i}$ is $1$. Then $a_{t,i}=J_{|C_i|}$. We then have $a_{t,i}d_{ii}=J_{|C_i|}d_{ii}$. Note $J_{|C_i|}d_{ii}=0$, so $a_{t,i}d_{ii}=0$ in this case. If $C_k\neq C_i$, then $yx^{-1}\not\in C_k$ for any $x\in C_t$ and $y\in C_i$. Then the $(x,y)$ entry of $a_{t,i}$ is $0$ for all $x,y$. So $a_{t,i}$ is the all $0$ matrix. In this case $a_{t,i}d_{ii}=0$. Hence, in either of these cases $a_{t,i}d_{ii}=0$.

    {\bf Case 3: }Consider the block, $a_{t,i}$ of $A_k$ when $C_t=zG'$ for some $z\in G\setminus G'$. We consider subcases based on if $C_i=C_t$ or not.\\
    {\bf Subcase 3.1: } $C_t=C_i$.  Consider the $(x,y)$ entry of $a_{t,i}$. We have $x\in C_t=zG'$ and $y\in C_i=zG'$. Then $yx^{-1}\in (zG')^{-1}zG'=G'$. Then as in Subcase 2.1, if $a_{t,i}$ is nonzero in some entry then either $a_{t,i}d_{ii}=d_{ii}$ or $a_{t,i}d_{ii}=-d_{ii}$. If $a_{t,i}=0$, then $a_{t,i}d_{ii}=0$.\\
    {\bf Subcase 3.2: }Suppose $C_i\neq C_t$. Consider the $(x,y)$ entry of $a_{t,i}$. We have $x\in C_t$ and $y\in C_i$. If $C_i=G'\setminus \{e\}$, then $yx^{-1}\in z^{-1}G'$. If $C_i\neq G'\setminus \{e\}$, then $C_i=uG'$ for some $u\in G\setminus G'$. Then $yx^{-1}\in uz^{-1}G'$. Either way $yx^{-1}\in wG'$ for some $w\in G\setminus G'$. Thus, the $(x,y)$ entry of $a_{t,i}$ is nonzero if and only if $C_k=wG'$. In this case $yx^{-1}\in wG'$ for all $x\in C_t=zG'$ and $y\in C_i$. So every entry of $a_{t,i}$ is $1$. Then $a_{t,i}=J_{|C_i|}$, and $a_{t,i}d_{ii}=J_{|C_i|}d_{ii}=0$. If $C_k\neq wG'$, then $yx^{-1}\not\in C_k$ for any $x\in C_t$ and $y\in C_i$. Then the $(x,y)$ entry of $a_{t,i}$ is $0$ for all $x,y$. So $a_{t,i}$ is the all $0$ matrix. In this case $a_{t,i}d_{ii}=0$. Hence, in both cases $a_{t,i}d_{ii}=0$.

    From these three cases we have $a_{t,i}d_{ii}=0$ if $t\neq i$. For $t=i$, if $C_k=C_0$, then $a_{i,i}d_{ii}=d_{ii}$ and if $C_k=G'\setminus \{e\}$ then $a_{i,i}d_{ii}=-d_{ii}$. If $C_k\not\subseteq G'$, then $a_{i,i}d_{ii}=0$.
    Therefore
    \[A_kB_i=\left\{\begin{array}{cc}
        B_i & \text{ if }C_k=C_0 \\
        -B_i & \text{ if }C_k=G'\setminus \{e\} \\
        0 & \text{ otherwise }
    \end{array}\right..\]
    So $A_kB_i\in W_i$ for all $k$. Hence, $W_i$ is closed under multiplication on the left by the $A_k$ matrices. Hence, $W_i$ is a left ideal of $T(G)$. We note that $B_i$ is a symmetric matrix for all $i$. Then every matrix in $W_i$ is symmetric, which implies $W_i$ is a right idea as well. Thus, $W_i$ is two-sided ideal of $T(G)$ for all $1\leq i\leq p^n-1$. 
\end{proof}

\begin{Cor}\label{cor:frobid}
    Let $G=Z_p^{n}\rtimes Z_{p^n-1}$ for some prime $p$ and $n\geq 1$. Suppose $C_0=\{e\},C_1,\cdots, C_{p^n-1}$ are the classes of $G$. Defining $B_i$ as in Proposition \ref{Prop:frob1dim}, we have that for each $1\leq i\leq p^n-1$, $\frac{|C_i|-1}{|C_i|}B_i$ is a non-primary primitive idempotent of $T(G)$. 
\end{Cor}

\begin{proof}
    We note that direct computation yields for all $1\leq i\leq p^n-1$, $B_i^2=(1+\frac{1}{|C_i|-1})B_i$, giving for each $1\leq i\leq p^n-1$, $\frac{|C_i|-1}{|C_i|}B_i$ is a non-primary primitive idempotent of $T(G)$.
\end{proof}

\begin{Cor}\label{cor:frobdecomp}
     Let $G=Z_p^{n}\rtimes Z_{p^n-1}$ for some prime $p$ and $n\geq 1$. Suppose $C_0=\{e\},C_1,\cdots, C_{p^n-1}$ are the conjugacy classes of $G$. Then the Wedderburn decomposition of $T(G)$ is
     \[T(G)\cong V\oplus W_1\oplus W_2 \oplus \cdots \oplus W_{p^n-1}.\]
\end{Cor}

\begin{proof}
   Recall $V$ is an irreducible two sided ideal of $T(G)$ with dimension equal to the square of number of conjugacy classes of $G$. From the proof of Theorem \ref{cor:frobdim}, $G$ has $p^n$ classes. Then $V$ has dimension $p^{2n}$. From Proposition \ref{Prop:frob1dim}, $W_i$ is a $1-$dimensional ideal of $T(G)$ for all $1\leq i\leq p^n-1$. We note that since the row sum of each row of the basis element of $W_i$ is $0$, that each $W_i$ is orthogonal to $V$ under the Hermitian inner product. As the basis elements for each $W_i$ are nonzero in different blocks, $W_i$ and $W_j$ are orthogonal for all $i\neq j$.
    Then $\dim(V\oplus W_1\oplus W_2\oplus \cdots \oplus W_{p^n-1})=p^{2n}+p^{n}-1$.
    As each of the $W_i$ and $V$ are ideals of $T(G)$ we have $V\oplus W_1\oplus W_2\oplus \cdots \oplus W_{p^n-1}\subseteq T(G)$.
    From Theorem \ref{cor:frobdim}, $\dim(T(G))=p^{2n}+p^n-1$. Then $V\oplus W_1\oplus W_2\oplus \cdots \oplus W_{p^n-1}$ is a subalgebra of $T(G)$ with the same dimension. We thus have $T(G)\cong V\oplus W_1\oplus W_2 \oplus \cdots \oplus W_{p^n-1}$.
\end{proof}

\bibliographystyle{plain}
\bibliography{Sources}

\end{document}